\documentclass[a4paper,11pt,leqno]{amsart}

\usepackage{amsmath,amsfonts,amssymb}
\usepackage{pifont,graphicx,epsfig}
\usepackage[bf]{caption2}
\usepackage{epstopdf}

\newtheorem{Theorem}{Theorem}

\def\P{\hbox{\font\dubl=msbm10 scaled 1100 {\dubl P}}}
\def\N{\hbox{\font\dubl=msbm10 scaled 1100 {\dubl N}}}
\def\sN{\hbox{\font\dubl=msbm10 scaled 800 {\dubl N}}}

\title[Deleting Digits]{Deleting Digits}

\author[Ioulia N. Baoulina, Martin Kreh, J\"orn Steuding]{Ioulia N. Baoulina, Martin Kreh, J\"orn Steuding}

\date{}

\begin{document}

\maketitle

\section{A Paper by Shallit on Primes}

We consider here the positive integers with respect to their unique decimal expansions, where each $n\in\mathbb N$ is given by $n=\sum_{j=0}^k \alpha_j10^{j}$ for some non-negative integer $k$ and digit sequence $\alpha_k\alpha_{k-1}\ldots\alpha_0$. With slight abuse of notation, we also use $n$ to denote $\alpha_k\alpha_{k-1}\ldots\alpha_0$. For such sequences of digits (as well as for the numbers represented by the corresponding expansions) we write $x\triangleleft y$ if $x$ is a subsequence of $y$, which means that either $x=y$ or $x$ can be obtained from $y$ by deleting some digits of $y$. For example, $514\triangleleft 352148$. The main problem is as follows: {\it given a set $S\subset \N$, find the smallest possible set $M\subset S$ such that for all $s\in S$ there exists $m\in M$ with $m \triangleleft s$.} As a matter of fact, the set 
$$
M(S):=\{s\in S\,\mid\, \{n\in S\,:\,n<s, n\triangleleft s\}=\varnothing\}
$$ 
solves this problem; clearly, every element of $S$ contains some element of $M(S)$ as a subword. The set $M(S)$ is said to be the {\it minimal set} and its elements are also called {\it minimal}. 
\par

Remarkably, {\it $M(S)$ is finite for every possible $S\subset \N$}. This follows from a celebrated theorem from the theory of formal languages, the so-called lemma of Higman \cite{higman} (see also the Chapter entitled {\it Subwords} by Jacques Sakarovitch and Imre Simon \cite{lothaire} in Lothaire's  encyclopaedia {\it Combinatorics on Words}).
\par

In general, it seems to be difficult to compute $M(S)$ for a given $S$, the known proofs of Higman's lemma are ineffective. In the case of the set of prime numbers $\P=\{2,3,5,\ldots\}$ Jeffrey Shallit \cite{shallit} succeeded in determining $M(\P)$ by elementary means, namely
\begin{eqnarray*}
M(\P)&=&\{2,3,5,7,11,19,41,61,89,409,449,499,881,991,6469,\\
&& \ 6949,9001,9049,9649,9949,60649,666649,946669,\\
&& \ 60000049,66000049,66600049\}. 
\end{eqnarray*}
Already by a first glimpse on the elements of $M(\P)$ one can guess the kind of reasoning necessary to prove this explicit form for the minimal set of primes. However, in the case of the set $2^{\sN}=\{2^n\,:\,n\in\N_0\}$ consisting of the powers of $2$ Shallit \cite{shallit} conjectured
$$
M(2^{\sN})=\{1,2,4,8,65536\}.
$$
\pagebreak\newpage\noindent
Moreover, he observed that this is true if every number $16^m$ with $m\geq 4$ has at least one digit from $\{1,2,4,8\}$. This type of problem seems to be difficult to tackle. Of course, we know that the powers of $2$ are distributed according to Benford's law (i.e., the sequence $n\log_{10}2$ is uniformly distributed modulo one),  nevertheless, this just indicates - in an appropriate probabilistic framework - that a power of $2$ without any digit from $\{1,2,4,8\}$ must be a very rare event.
\par

In this note we present our new results on the minimal sets of a few other arithmetically interesting sets of positive integers (in Section 2 \& 3), we furthermore address questions on size and shape of minimal sets in general and indicate the absence of structure behind (Section 4), and finally we conclude with another example, namely the hypothetical minimal set for perfect numbers and analogous question for other bases (Section 5). Our main aim is to make the beautiful theorem of Higman and Shallit's paper more popular in order to attract more research in this direction.

\section{Sums of Squares and Quadratic Residues}

In this section we look at three subsets of $\N$ for which the minimal sets can be determined easily. Whereas in the first example the reason for this is that the examined set contains many digits, in the second and third examples we exploit the fact that the elements of the sets are in some sense well-distributed. For the first example, this means that we have almost all digits lying in the set, which means that there are just few numbers that we are left to examine. In the other two examples, the elements of the examined sets lie in certain residue classes for some modulus, which will be helpful. Other explicit constructions of minimal sets of residue classes and some discussions about the structure can be found in \cite{kreh}.

\begin{Theorem}
Let 
\begin{equation*}
\fbox{$3$} := \lbrace n \in \N\,\mid\, \exists x,y,z \in \N_0 : n = x^2+y^2+z^2 \rbrace.
\end{equation*}
Then
\begin{equation*}
M(\fbox{$3$}) = \lbrace 1,2,3,4,5,6,8,9,70,77 \rbrace.
\end{equation*}
\end{Theorem}

\begin{proof}
It is easy to see that each of the numbers $1,2,3,4,5,6,8,9,70,77$ can be written as a sum of three squares, so they belong to $\fbox{3}$. Clearly $1,2,3,4,5,6,8,9$ belong to $M(\fbox{3})$. Since $7 \notin \fbox{3}$, $70$ and $77$ also belong to $M(\fbox{3})$. If $n \in \fbox{3}$ is different from $1,2,3,4,5,6,8,9,70,77$, then either $d \triangleleft n$ for $d \in \lbrace 1,2,3,4,5,6,8,9 \rbrace$ and $n$ cannot belong to $M(\fbox{3})$ or $n$ consists only of digits $0$ and $7$. In the latter case either $70 \triangleleft n$ or $77 \triangleleft n$, and so $n\notin M(\fbox{3})$.
\end{proof}

Two strings of digits $x$ and $y$ are called {\it incomparable} if neither $x\triangleleft y$ nor $y\triangleleft x$. The reader might have noticed that a subset of $S$ consisting of pairwise incomparable elements is in general not contained in $M(S)$ (as the example $S=\N$ and the subset of all two digit numbers shows). However, in the previous proof for any element $x$ in the set $\{1,2,3,4,5,6,8,9,70,77\}$ of pairwise incomparable elements, there was no $y \in \fbox{3}, y \neq x$ with $y \triangleleft x$. In this case the set of pairwise incomparable elements is indeed contained in $M(S)$. We shall use this idea quite often in the sequel.

\begin{Theorem}
Let 
\begin{equation*}
\Box\bmod\,6 := \lbrace n \in \N\,\mid\, \exists x \in \N : x^2 \equiv n \mod 6 \rbrace.
\end{equation*}
Then
\begin{equation*}
M(\Box\bmod\,6) = \lbrace 1,3,4,6,7,9,22,25,28,52,55,58,82,85,88 \rbrace.
\end{equation*}
\end{Theorem}

\begin{proof}
It is easy to check that $n \in \Box\bmod\,6$ if and only if $n \equiv 1,3,4,6 \mod 6$. So we have indeed $\lbrace 1,3,4,6,7,9,22,25,28,52,55,58,82,85,88 \rbrace \subset \Box\bmod\,6$. Moreover, these elements are pairwise incomparable. Let $n \in \Box\bmod\,6$ be arbitrary. If $d \triangleleft n$ with $d \in \lbrace 1,3,4,6,7,9 \rbrace$, we are done. So suppose that $n$ contains none of these digits. Assume first that $n$ contains exactly two digits. Then $22,25,28,52,55,58,82,85,88 \in \Box\bmod\,6$ and $20,50,80 \notin \Box\bmod\,6$. If $n$ has more than two digits, then either two of its digits are nonzero and so there is an $x \in M(\Box\bmod\,6)$ with $x<n$, $x \triangleleft n$, or $n$ has only one nonzero digit. In the latter case, this digit is congruent to $2$ modulo $3$, hence $n \equiv 2 \mod 6$, and therefore $n \notin \Box\bmod\,6$.
\end{proof}

\begin{Theorem}
Let 
\begin{equation*}
\Box\bmod\,7 := \lbrace n \in \N\,\mid\,\exists x \in \N : x^2 \equiv n \mod 7 \rbrace.
\end{equation*}
Then
\begin{equation*}
M(\Box\bmod\,7) = \lbrace 1,2,4,7,8,9,30,35,36,50,53,56,60,63,65,333,555,666 \rbrace.
\end{equation*}
\end{Theorem}

\begin{proof}
Since $n \in \Box\bmod\,7$ if and only if $n \equiv 1,2,4,7 \mod 7$, we have $\lbrace 1,2,4,7,8,9,30,35,36,50,53,56,60,63,65,333,555,666 \rbrace \subset \Box\bmod\,7$; again these elements are pairwise incomparable. Let $n \in \Box\bmod\,7$. If $d \triangleleft n$ with $d \in \lbrace 1,2,4,7,8,9 \rbrace$, we are done. So suppose that $n$ contains none of these digits. If $n$ has exactly two digits, then $30,35,36,50,53,56,60,63,65 \in \Box\bmod\,7$ and $33,55,66 \notin \Box\bmod\,7$. If $n$ has at least three digits and at least two of them are distinct we are done. It therefore remains to assume that $n$ is of the form $dd\ldots dd$ with $d \in \lbrace 3,5,6 \rbrace$. Then either $333 \triangleleft n$ or $555 \triangleleft n$ or $666 \triangleleft n$.
\end{proof}

\section{Values of Arithmetical Functions}

In this section, we consider sets of the type $S=\{n \in \N\,\mid\, \exists x \in \N : f(x)=n\}$ where $f:\N \rightarrow \N$ is an arithmetical function. Recall that the Euler function $\varphi(n)$ is defined to be the number of positive integers not exceeding $n$ which are relatively prime to $n$. If $n=\prod_{p|n} p^{k_p}$ is the unique prime factorization of $n$, then $\varphi(n)=\prod_{p|n} (p-1) p^{k_p-1}$.

\begin{Theorem}
Let
\begin{equation*}
\varphi(\N) := \lbrace n \in \N\,\mid\, \exists x \in \N : \varphi(x)=n \rbrace.
\end{equation*}
Then
\begin{equation*}
M(\varphi(\N)) = \lbrace 1, 2, 4, 6, 8, 30, 70, 500, 900, 990, 5590, 9550, 555555555550\rbrace.
\end{equation*}
\end{Theorem}

\begin{proof}
Observe that the numbers above are pairwise incomparable. Further, as
\begin{align*}
\varphi(1)&=1,&\varphi(3)&=2,&\varphi(5)&=4,&\varphi(7)&=6,\\
\varphi(16)&=8,&\varphi(31)&=30,&\varphi(71)&=70,&\varphi(625)&=500,\\
\varphi(1057)&=900,&\varphi(991)&=990,&\varphi(5591)&=5590,&\varphi(9551)&=9550,
\end{align*}
and
$$
\varphi(555555555551)=555555555550,
$$
we see that $\lbrace 1, 2, 4, 6, 8, 30, 70, 500, 900, 990, 5590, 9550, 555555555550\rbrace\subset \varphi(\N)$. Note also that $3,5,7,9\notin \varphi(\N)$. Now assume that $n\in \varphi(\N)$ has at least two digits. If there exists $d\in\{1,2,4,6,8\}$ such that $d \triangleleft n$ then $n\notin M(\varphi(\N))$. Suppose that $n$ contains only the digits $0,3,5,7,9$. Since $\varphi(m)$ is even for $m>2$, the last digit is $0$. It is easy to see that $50,90 \notin \varphi(\N)$. Now assume that $n$ has at least three digits. If $n$ contains the digits $3$ or $7$, then $30 \triangleleft n$ or $70 \triangleleft n$, respectively. Consequently, $n\notin M(\varphi(\N))$. Next assume that $n$ contains only the digits $0,5,9$. It is easily verified that $550,590,950\notin \varphi(\N)$. Suppose that $n$ has at least four digits. Then we have the following possibilities:
\begin{itemize}
\item $n$ contains at least two zeros. Then either $500 \triangleleft n$ or $900 \triangleleft n$.

\item $99 \triangleleft n$. Then $990 \triangleleft n$.

\item $559 \triangleleft n$. Then $5590 \triangleleft n$.

\item $955 \triangleleft n$. Then $9550 \triangleleft n$.

\item $n=5950$. Then $n\notin \varphi(\N)$.

\item $n=\underbrace{55\ldots 5}_{\ell}0$ with $\ell\ge 3$.
\end{itemize}
Therefore, it remains to show that if $n=\underbrace{55\ldots 5}_{\ell}0$ with $3\le\ell\le 10$, then $n\notin \varphi(\N)$. To this end, assume that $\varphi(m)=n$ for some $m\in\N$. Since $4\nmid n$, we have $4\nmid m$ and $m$ has exactly one odd prime divisor, that is $m=p^k$ or $m=2p^k$ with $k\in\N$ and odd $p\in\P$. In both cases we must have $p^{k-1}(p-1)=n$. Using a computer algebra package, we find the unique prime factorization of each of the values of $n$ under consideration, namely,
\begin{align*}
5550&=2\cdot3\cdot5^2\cdot37, & 55550&=2\cdot5^2\cdot11\cdot101,\\
555550&=2\cdot5^2\cdot41\cdot271, & 5555550&=2\cdot3\cdot5^2\cdot7\cdot11\cdot13\cdot37,\\
55555550&=2\cdot5^2\cdot239\cdot4649, & 555555550&=2\cdot5^2\cdot11\cdot73\cdot101\cdot137,\\
5555555550&=2\cdot3^2\cdot5^2\cdot37\cdot333667, & \!\!\!\!55555555550&=2\cdot5^2\cdot11\cdot41\cdot271\cdot9091.
\end{align*}
We see that if $p^{k-1}\mid n$ with $k\ge 3$, then $k=3$ and $p=3$ or $5$, and so $p^{k-1}(p-1)\ne n$. Further, it is easy to check that there is no prime number $p$ with $p(p-1)=n$. Hence $k=1$ and $n+1$ is a prime. However, all the numbers $\underbrace{55\ldots 5}_{\ell}1$, $3\le\ell\le 10$, are composite, and this gives the desired contradiction.
\end{proof}

It is not difficult to find minimal sets for some sets of shifted values of the Euler function. For example, it can be proved that
$$
M(3+\varphi(\N))=\lbrace 4, 5, 7, 9, 11, 13, 21, 23, 31, 33, 61, 63, 81, 83\rbrace,
$$
where
$$
3+\varphi(\N):= \lbrace n \in \N\,\mid\, \exists x \in \N : 3+\varphi(x)=n \rbrace.
$$
We leave this as an exercise for the interested reader.

Related to Euler's function but less well-known is the Dedekind $\psi$-function, for $n = \prod_{p\mid n} p^{k_p}$ defined by
\begin{equation*}
\psi(n) := \prod_{p\mid n} (p+1) p^{k_p-1}.
\end{equation*}
As we shall see in the proof of the following theorem, the determination of the corresponding minimal set is possible, although requires a little more effort than in the case of Euler's $\varphi$.

\begin{Theorem}
Let $\psi$ be the Dedekind $\psi$-function and
\begin{equation*}
\psi(\N) := \lbrace n \in \N\,\mid\, \exists x \in \N : \psi(x)=n \rbrace.
\end{equation*}
Then
\begin{align*}
M(\psi(\N)) = \lbrace &1,3,4,6,8,20,72,90,222,252,500,522,552,570,592,750,770, \\
&992,7000,\underbrace{55\ldots 5}_{69}0\rbrace.
\end{align*}
\end{Theorem}

\begin{proof}
First note, that the elements given above are pairwise incomparable and we have
$$
\begin{array}{rclrclrclrcl}
\psi(1)&=&1, &\psi(2)&=&3, &\psi(3)&=&4, &\psi(5)&=&6,\\
\psi(7)&=&8, &\psi(19)&=&20, &\psi(71)&=&72, &\psi(89)&=&90,\\
\psi(146)&=&222, &\psi(251)&=&252, &\psi(499)&=&500, &\psi(521)&=&522, \\
\psi(411) &=&552, &\psi(569)&=&570, &\psi(511)&=&592, &\psi(625)&=&750,\\
\psi(769)&=&770, &\psi(991)&=&992, &\psi(6631)&=&7000&&&
\end{array}
$$
and
$$
\psi(\underbrace{55\ldots 5}_{68}49)=\underbrace{55\ldots 5}_{69}0,
$$
since $\underbrace{55\ldots 5}_{68}49$
is a prime, all these numbers belong to $\psi(\N)$. 
\par

Now let $n \in \psi(\N)$ be arbitrary with at least two digits. If $d \triangleleft n$ with $d \in\{1,3,4,6,8\}$, then $n \notin M(\psi(\N))$. So $n$ does only contain the digits $2,5,7,9,0$. It follows from the product formula $\psi(m) = \prod_{p\mid m} (p+1) p^{k_p-1}$ defining $\psi$ that $\psi(m)$ is even for $m>2$. First we note  that $22, 50, 52, 70, 92,$ $292, 502, 550, 700, 922, 952 \notin \psi(\N)$. Now suppose that $n$ has at least four digits. Then one of the following cases holds:
\begin{itemize}
\item $n$ ends with a $2$ and contains a $7$. Then $72 \triangleleft n$.
\item $n$ ends with a $2$, contains a zero and does not contain a 7. Then $20 \triangleleft n$ or $90 \triangleleft n$ or $502 \triangleleft n$. In the latter case $n$ contains one of the following strings:
\begin{equation*}
5002, 2502, 5202, 5022, 5502, 5052,9502, 5902, 5092.
\end{equation*} 
Then $20 \triangleleft n$ or $90 \triangleleft n$ or $500 \triangleleft n$ or $522 \triangleleft n$ or $552 \triangleleft n$ or $592 \triangleleft n$.
\item $n$ ends with a $2$ and contains only the digits $2$, $5$ and $9$. Then $222 \triangleleft n$ or $252 \triangleleft n$ or $522 \triangleleft n$ or $552 \triangleleft n$ or $992 \triangleleft n$.
\item $n$ ends with a zero and contains a $2$ or a $9$. Then $20 \triangleleft n$ or $90 \triangleleft n$.
\item $n$ ends with a zero and contains a $7$ and a $5$. Then $570 \triangleleft n$ or $750 \triangleleft n$.
\item $n$ ends with a zero and contains only the digits $7$ and $0$. Then $770 \triangleleft n$ or $7000 \triangleleft n$.
\item $n$ ends with a zero and contains only the digits $5$ and $0$. Then $500 \triangleleft n$ or $550 \triangleleft n$.
\end{itemize}
\smallskip

Thus, the only numbers $n$ that need to be examined are of the form $5 \ldots 50$. It suffices to show that if $n=\underbrace{55\ldots 5}_{\ell}0$ with $3\le\ell\le 68$, then $n\notin \psi(\N)$.

We note, that if $m$ has $k$ distinct odd prime factors, then $2^k\mid\psi(m)$. In our case, $n \equiv 2 \mod 4$, so if there is $m \in \N$ with $\psi(m)=n$, then $m = p^k$ or $m=2 p^k$ with $p \in \P$ odd. Hence either $p^{k-1}(p+1)=n$ or $3p^{k-1}(p+1)=n$ with $k\in\N$ and odd $p\in\P$. Note that $4\nmid(p+1)$. A computer search shows that if $p^{k-1}\mid n$ with $k\ge 3$ and $4\nmid(p+1)$, then $k=3$ and $p=5$. This implies that either $n=150$ or $n=450$, which is a contradiction. Next, one can easily verify that there is no prime $p$ with $n=p(p+1)$ or $3p(p+1)$, and so $k\ne 2$. Finally, note that neither $n-1$ nor $(n/3)-1$ is a prime, which yields $k\ne 1$. This concludes the proof.
\end{proof}

\section{The Absence of Structure}

The examples from the previous sections already indicate that minimal sets behave rather unexpectedly. It is easy to obtain some sets in explicit form, while for other sets of {\it higher arithmetical complexity} the explicit form of the corresponding minimal set may be only achieved conditionally. Already the size of minimal sets seems to be almost unpredictable.
\par

For $k\in\N$ define $S_k:=\N\cap [10^{k-1},10^k)$ (the set of positive integers with $k$ digits in base $10$). One easily verifies $M(S_k)=S_k$ which shows that the minimal set can be as large as we please even for finite sets. Moreover, $S_k$ is minimal among all sets $S$ with $\sharp M(S)=9\cdot 10^{k-1}$. In general it seems to be difficult to prove an upper bound for the number of minimal elements in an arbitrary set. 
\par

One might be tempted to try to build up some kind of set theory for minimal sets. The following statement is trivial:
$$
M(S\cup T)\subset M(S)\cup M(T).
$$
The left-hand side above can be equal to the right-hand side (if, for example, $S=2+10\N_0$ and $T=3+10\N_0$) and smaller (if, for example, $S=\{2\}\cup (\P \cap (1+4\N_0))$ and $T=\P \cap (3+4\N_0$)). The same inclusion with $\cap$ in place of $\cup$ on the right hand-side is in general false as follows from the example $S=\P$ and $T=1+4\N_0$. Moreover, even $M(S\cap T)\subset M(S)\cup M(T)$ is false as the counterexample $S=7+10\N$ and $T=\P$ shows. 
\par

Probably, the main obstacle in order to prove something 'structural' (whatever that means) is that $S\mapsto M(S)$ is not monotone: e.g., $\P\subset \N$ but $\sharp M(\P)>\sharp M(\N)$. We conclude that minimal sets indeed behave erratically (but would be glad if the reader could disprove our non-mathematical statement).

\section{An Odd End}

Choosing another base than $10$ would not make things easier in general (see \cite{BDS} for recent results in this direction). Of course, with the binary expansion in place of the decimal expansion a few minimal sets would look more simple, e.g., $M(\P)=\{10_2,11_2\}$. However, there are still plenty of sets, interesting from a  number-theoretical point of view, where the corresponding minimal set seems to be difficult to describe. 
\par

For instance, the set of perfect numbers ${\sf Perfect}=\{6,28,496,\ldots\}$ is conjectured to contain no odd numbers. Recall that a positive integer is called {\it perfect} if it equals the sum of its proper divisors. As already Euclid knew (in different language), every Mersenne prime $2^p-1$ (with prime $p$) yields an even perfect number $2^{p-1}(2^p-1)$, and Leonhard Euler proved that every even perfect number is of this form (see Dickson \cite{dickson} for this and further details about and the history of perfect numbers). \'Edouard Lucas \cite{lucas} showed that every even perfect number different from $6$ and $496$ ends with decimal digits $16,28,36,56$ or $76$. Hence, {\it if there is no odd perfect number, then the minimal set of perfect numbers is given by $M({\sf Perfect})=\{6,28\}$}. If there exists an odd perfect number, although much can be said about its hypothetical multiplicative structure, we cannot exclude the possibility that this odd number is a minimal element. In base $2$ we would have $M({\sf Perfect})=\{110_2\}$ if there is no odd perfect number, and we leave the computation of the minimal set with respect to other bases (under the same assumption) to the interested reader. 

\section*{Acknowledgment}
The authors thank the referee for careful reading of the manuscript and many helpful suggestions.

\bigskip

\noindent {\footnotesize Ioulia N. Baoulina, Department of Mathematics, Moscow State Pedagogical University, Krasnoprudnaya str. 14, Moscow 107140, Russia, jbaulina@mail.ru}

\smallskip

\noindent {\footnotesize Martin Kreh, Department for Algebra and Number Theory, University of Hildesheim, Samelsonplatz 1, 31141 Hildesheim, Germany, kreh@imai.uni-hildesheim.de}

\smallskip

\noindent {\footnotesize J\"orn Steuding, Department of Mathematics, W\"urzburg University, Emil-Fischer-Str. 40, 97\,074 W\"urzburg, Germany, steuding@mathematik.uni-wuerzburg.de}

\end{document}